\documentclass[10pt,a4paper]{article}
\usepackage[utf8]{inputenc}
\usepackage{amsmath}
\usepackage{amsfonts}
\usepackage{amssymb}
\usepackage{amsthm}
\usepackage{verbatim}
\usepackage{graphicx}
\usepackage{subfigure}

\newtheorem{theorem}{Theorem}[section]
\newtheorem{lemma}[theorem]{Lemma}
\newtheorem{proposition}[theorem]{Proposition}
\newtheorem{corollary}[theorem]{Corollary}

\theoremstyle{definition}
\newtheorem{definition}[theorem]{Definition}
\newtheorem{example}[theorem]{Example}

\newcommand{\ggl}{{\rm Gl}(d,\mathbb{R})}

\newcommand{\ww}[1]{\widehat{#1}}

\newcommand{\oo}{\mathcal{O}}

\newcommand{\intr}{\rm int}

\title{Existence of eigensets on bilinear control systems.}

\author{Eduardo Celso Viscovini\thanks{Departamento de Matem\'{a}tica, Universidade Estadual de Maring\'{a}. \textit{%
e-mail: eduardocelsoviscovini@gmail.com }. Supported by CAPES grant n.
88887.669835/2022-00}}

\begin{document}

\maketitle

\begin{abstract}
    For bilinear control systems in $\mathbb{R}^d$ we prove, under an accessibility hypothesis, the existence of a nontrivial compact set $D\subset\mathbb{R}^d$ satisfying $\oo_t(D)=e^{tR}D$ for all $t>0$, where $R\in\mathbb{R}$ is a fixed constant and $\oo_t(D)$ denotes the orbit from $D$ at time $t$. This property generalizes the trajectory of an eigenvector on a linear dynamical system, and merits such a set the name "eigenset".
\end{abstract}

\section{Introduction}
Consider the bilinear control system $\Sigma$ on $\mathbb{R}^d$ defined by
\begin{equation}
\label{systemEquation}
    \dot x=Ax+\sum_{i=1}^mu_iB_ix,
\end{equation}
where $A,B_1,B_2,...,B_m\in M_{d\times d}(\mathbb{R})$, $U\subset\mathbb{R}^m$ is a compact, convex set, and $u=(u_1,u_2,...,u_m)\in\mathcal{U}=\{f:\mathbb{R}\rightarrow U;\text{ locally integrable}\}$.

This system induces a system $\ww{\Sigma}$ on the unit sphere through projection (this construction is described in more details in Section \ref{preliminaries}). In this paper we prove that if $\ww{\Sigma}$ is accessible then there is a nontrivial compact set $D\subset\mathbb{R}^d$ with the property
\begin{equation}
    \label{eigenvectorLikeProperty}
    \oo_t(D)=e^{tR}D
\end{equation}
for all $t>0$, where $R\in\mathbb{R}$ is a fixed value and $\oo_t(D)$ denotes the set of points that can be reached from $D$ at time exactly $t$ using trajectories from the original system $\Sigma$. If $\Sigma$ is accessible then so is $\ww{\Sigma}$, thus this result includes any bilinear system that is accessible in $\mathbb{R}^d$.

Bilinear control systems have been broadly studied in the theory of control systems, having the simplest definition after the linear control systems. Notably, their controllability remains an open question despite this simplicity. Nonetheless, several partial results exist. Kučera \cite{Kučera} proved that if $A=0$ and $0\in \intr(U)$ then the system $\Sigma$ is controllable if, and only if, the Lie group generated by exponentials of the form $e^{tB_i};i\in\{1,2,...,m\},t\in\mathbb{R}$ is transitive in $\mathbb{R}^d$, and Boothby \cite{Boothby} characterized the connected Lie subgroups of $\ggl$ with this property. Barros et al \cite{BarrosFilhoDoRocíoAndSanMartin} gave necessary and sufficient conditions for the controllability of bilinear control systems in $\mathbb{R}^2$. Jurdjevic and Kupka \cite{JurdjevicKupka} and Sachkov \cite{Sachkov} show several results on controllability using the Lie saturate. Another direction that has yielded remarkable results is the study of invariant control sets on the flag manifolds of semisimple Lie groups, in special the flag type of the system semigroup (see San Martin \cite{SanMartin} and \cite{SanMartin2}). If the system's group coincides with $\rm{Sl}(n,\mathbb{R})$ (or, equivalently, if the Lie algebra generated by $A,B_1,B_2,...,B_m$ coincides with $\rm{sl}(n,\mathbb{R})$), do Rocio et al. \cite{doRocioSanMartinSantana} showed that the existence of invariant convex cones and invariant convex sets for the bilinear control system depends only on the flag type of its semigroup, and in \cite{doRocioSanMartinSantana2} they extend these results to other transitive groups. For more on bilinear control systems and on control theory in general, refer to \cite{dynamicsOfControl,Elliot,Sachkov} and references therein.

The results in this paper move on a new direction. Eigenvectors have been widely successful in describing linear transformations, and, by extension, linear dynamical systems. Recall that a linear endomorphism $M\in M_{d\times d}(\mathbb{R})$ induces a dynamical system by
$$\dot x(t)=Mx(t),$$
and the respective flow is given by $\varphi:\mathbb{R}\times\mathbb{R}^d\rightarrow\mathbb{R}^d,\varphi(t,x)=e^{tM}x$. If $x\in\mathbb{R}^d$ is nonzero, then $x$ is an $M$-eigenvector if, and only if, there is $r\in\mathbb{R}$ such that 
$$\varphi(t,x)=e^{tr}x$$
for all $t\in\mathbb{R}$. Thus, $x$ is an Eigenvector if, for all $t\in\mathbb{R}$, the point reached at time $t$ from the initial condition $x(0)=x$ is a multiple of $x$.

In the case of a bilinear control system, it doesn't make sense to ask this same property for a point, as the orbit of a point $x$ at time $t$ under the control system becomes a set. Instead,
the property \ref{eigenvectorLikeProperty} appears as a generalization of this concept.

This paper is structured as follows. In section \ref{preliminaries} we recall some known results that will be needed on the next sections and give the formal definition of eigenset to be used in the paper. Section \ref{mainResults} contains the main results of this paper, and is divided in three subsections. Section \ref{theFunctionXi} defines the function $\xi$ and shows some of its properties. This function will be used on the later sections. Section \ref{sectionParticularCase} proves the main result for a particular case, and Section \ref{generalCase} extends the result for the general case.

\section{Preliminaries}
\label{preliminaries}
Consider the bilinear control system $\Sigma$ defined as in the previous section. We denote the solution (in the Carathéodory sense) of $\Sigma$ with initial condition $x$, control $u$ and time $t$ by $\varphi(t,x,u)$, or $\varphi_u^t(x)$ when convenient. The set $\mathcal{U}$ can be endowed with the weak* topology of $L_\infty(\mathbb{R},\mathbb{R}^m)=L_1(\mathbb{R},\mathbb{R}^m)^*$. Using this topology, one has that the solution $\varphi$ is continuous $\mathbb{R}\times\mathbb{R}^d\times\mathcal{U}$, and the hypothesis that $U$ is convex and compact implies that $\mathcal{U}$ is metrizable and compact (see \cite[Proposition 4.1.1]{dynamicsOfControl}).

The solution is linear on the variable $x$, that is,
$$\varphi(t,\alpha x+y,u)=\alpha\varphi(t,x,u)+\varphi(t,y,u)$$
for all $\alpha\in\mathbb{R}$ and $x,y\in\mathbb{R}^d$. Equivalently, $\varphi_u^t\in{\rm Gl}(d,\mathbb{R})$ for all $u\in\mathcal{U}$ and $t\in\mathbb{R}$. A consequence of this is that $\Sigma$ projects into a system in the sphere, defined by
$$\widehat\varphi(t,\pi(x),u)=\pi(\varphi(t,x,u))$$
where $\pi$ denotes the projection
$$\pi:\mathbb{R}^d\setminus\{0\}\rightarrow S^{d-1}$$
$$x\rightarrow \pi(x)= \frac{x}{\Vert x\Vert}.$$
We will denote this system by $\widehat{\Sigma}$, and use the notation $\widehat{x}=\pi(x)$ when convenient.

For $t\in\mathbb{R}$ and $x\in\mathbb{R}^d$ denote
$$\oo_t(x)=\{\varphi(x,u,t);u\in\mathcal{U}\}$$
the orbit of $x$ at time $t$. For $I\subset\mathbb{R}$ denote
$$\oo_I(x)=\bigcup_{t\in I}\oo_t,$$
and
$$\oo^+(x)=\oo_{[0,+\infty)}(x),$$
$$\oo^-(x)=\oo_{(-\infty,0]}(x).$$
One can also define the orbit at time $t$ of a set $C\subset\mathbb{R}$:
$$\oo_t(C)=\bigcup_{x\in C}\oo_t(x),$$
and $\oo_I(C),\oo^+(C),\oo^-(C)$ are defined similarly. A set $C$ is said to be invariant if it satisfies $\oo^+(C)\subset C$. We use these same notations for the system $\ww{\Sigma}$, but with $\widehat\oo$.

For $x,y\in\mathbb{R}^d$ and $t,s\in\mathbb{R}$ recall that $x\in\oo_t(y)$ if, and only if, $y\in\oo_{-t}(x)$. Furthemore, if $ts>0$ then $\oo_t(\oo_s(x))=\oo_{t+s}(x)$. These same properties also hold for $\widehat\oo$.

Throughout the paper we will assume that the system $\ww{\Sigma}$ is accessible, that is, for all $x\in S^{d-1}$ both $\ww{\oo}^+(\ww{x})$ and $\ww{\oo}^-(\ww{x})$ have nonempty interior. The results presented here apply, in particular, to the case when $\Sigma$ is accessible, as accessibility of $\Sigma$ implies accessibility of $\ww{\Sigma}$. $\ww{\Sigma}$ is analytic and, thus, accessibility is equivalent to the Lie algebra rank condition and to local accessibility, meaning that $\oo_{[0,\epsilon]}(x)$ and $\oo_{[-\epsilon,0]}(x)$ have nonempty interior for all $x\in S^{d-1}$ and $\epsilon>0$ (see Sussman and Jurdjevic \cite{SussmanAndJurdjevic}, and \cite[Theorem A.4.6]{dynamicsOfControl}). This yields the following Proposition, which is a direct consequence of Theorem 3.1.5 from \cite{dynamicsOfControl}.

\begin{proposition}
\label{propExistenceOfInvControlSet}
    If $\ww{\Sigma}$ is accessible then for all $\widehat x\in S^{d-1}$ there is at least one invariant control set $C\subset\overline{\widehat\oo^+(\widehat x)}$. $C$ is connected, satisfies $C=\overline{\intr(C)}$, and $\intr(C)\subset\widehat\oo^+(\widehat y)$ for all $\widehat y\in C$.
\end{proposition}
\begin{proof}
    The set $S^{d-1}$ is compact and invariant for $\widehat{\Sigma}$, and, as mentioned previously, the hypothesis of accessibility for $\ww{\Sigma}$ also implies local accessibility. The Theorem then follows from \cite[Theorem 3.1.5]{dynamicsOfControl}
\end{proof}

In the later sections we will also use the concept of a star set. The usual definition of a star set allows it to have a center in any point in $\mathbb{R}^d$, but, for the purposes of this paper, we restrict ourselves to star sets that are centered in $0$:

\begin{definition}
\label{definitionStarSet}
A set $W\subset\mathbb{R}^d$ is called a star set if for all $w\in W$ and $\alpha\in[0,1]$, $\alpha w\in W$.
\end{definition}

We also introduce the following definition:

\begin{definition}
    Let $D\subset\mathbb{R}^d$ a compact set such that $D\not\subset\{0\}$. We say $D$ is an eigenset for system $\Sigma$ if there is $R\in\mathbb{R}$ such that for all $t>0$
    $$\oo_t(D)=e^{tR}D.$$
\end{definition}

Our goal in this paper is to prove the existence of an eigenset $D$ for $\Sigma$.

\section{Main results}
\label{mainResults}
In this section we present the main results of this paper and the necessary definitions and Lemmas. Section \ref{theFunctionXi} defines the function $\xi$, which will be used in later results. Section \ref{sectionParticularCase} proves the existence of an eigenset on a particular case, and Section \ref{generalCase} extends this result for the general case.

\subsection{The function $\xi$}
\label{theFunctionXi}
We will start by defining the function $\xi$ in $\mathbb{R}^d\setminus \{0\}$, and showing some of its properties. This function plays a central role in the main results.

\begin{definition}
For $x\in\mathbb{R}^d\setminus\{0\}$, define $S_x\subset\mathbb{R}$ as the set of all $\alpha\in\mathbb{R}$ for which there are $u\in\mathcal{U}$ and $t>0$ satisfying
$$\varphi(t,x,u)=e^{t\alpha}x.$$
Define $\xi(x)=\sup(S_x)$, where the supremum is defined as $+\infty$ if the set is not bounded from above (although in Proposition \ref{propBoundedSystem} we will show that this never happens) and $-\infty$ if it is empty.
\end{definition}

Note that $\xi(x)=-\infty$ if and only if $\ww{x}\not\in\widehat\oo_{t>0}(\ww{x})$, that is, if $\ww{x}$ cannot be reached from itself at strictly positive time. Also, if $\ww{x}=\ww{y}$ then, by the linearity of the system, $S_x=S_y$ and $\xi(x)=\xi(y)$.

\begin{proposition}
\label{propBoundedSystem}
There is $r<\infty$ such that $\xi(x)\le r$ for all $x\in\mathbb{R}^d\setminus\{0\}$.
\end{proposition}
\begin{proof}
Since $\mathcal{U}$ is compact, the function
$$f:[1,2]\times\mathcal{U}\rightarrow \mathbb{R}$$
$$(t,x,u)\rightarrow \frac{1}{t}\ln(\Vert\varphi_u^t\Vert).$$
is bounded from above by some $r\in\mathbb{R}$, where the norm used above is the operator norm of $M_{d\times d}(\mathbb{R})$:
$$\Vert M\Vert=\sup_{x\in S^{d-1}}\Vert Sx\Vert=\sup_{x\in\mathbb{R}^d\setminus\{0\}}\frac{\Vert Sx\Vert}{\Vert x\Vert},$$
for all $M\in M_{d\times d}(\mathbb{R})$.

For any $t\ge 1$, $u\in\mathcal{U}$ and $\Vert x\Vert$, one can write $t=t_1+t_2+...+t_k$, with $t_1,t_2,...,t_k\in[1,2]$, and then, by the cocycle property,
$$\varphi_u^t=\varphi_{u_k}^{t_k}\varphi_{u_{k-1}}^{t_{k-1}}...\varphi_{u_1}^{t_1},$$
where $u_i\in\mathcal{U}$ is defined by
$$u_i(t)=u\left(t+\sum_{j=1}^{i-1}t_i\right).$$
By the definition of $r$, for all $i\in\{1,2,...,k\}$ we have
$$\ln(\Vert\varphi_{u_i}^{t_i}\Vert)\le t_ir.$$
Thus,
$$\ln(\Vert \varphi_u^t\Vert)\le \ln\left(\Vert\varphi_{u_k}^{t_k}\Vert\Vert\varphi_{u_{k-1}}^{t_{k-1}}\Vert...\Vert\varphi_{u_1}^{t_1}\Vert\right)=$$
$$\ln(\Vert\varphi_{u_k}^{t_k}\Vert)+\ln(\Vert\varphi_{u_{k-1}}^{t_{k-1}}\Vert)+...+\ln(\Vert\varphi_{u_1}^{t_1}\Vert)\le \sum_{i=1}^kt_ir=tr.$$
Therefore, this inequality still holds true for all $t\ge 1$.

Now let $x\in\mathbb{R}^d\setminus\{0\}$, and $\alpha\in S_x$. Then, there is $t>0$ and $u\in\mathcal{U}$ such that
$$\varphi(t,x,u)=e^{t\alpha}x.$$
Since this equality depends only on the values assumed by $u$ in the interval $(0,t)$, we can take $u$ as a cyclic function with period $t$. Then, the cocycle property combined with the linearity of the system then implies that
$$\varphi(nt,x,u)=e^{nt\alpha}x,$$
for all $n\in\mathbb{N}^*$. Taking $n$ sufficiently large so that $nt>1$, we have, by the previous argument,
$$\ln(\Vert \varphi_u^{nt}\Vert)\le ntr.$$
Thus,
$$\Vert e^{nt\alpha}x\Vert=\Vert\varphi(nt,x,u)\Vert\le e^{ntr}\Vert x\Vert,$$
which implies
$$e^{nt\alpha}\le e^{ntr}\Rightarrow \alpha<r.$$
Thus, $S_x$ is bounded from above by $r$, and, therefore, $\xi(x)\le r$.
\end{proof}

The next Lemma shows that we get the same values for $\xi$ if we look only at trajectories with time greater that some arbitrary $T$.

\begin{lemma}
\label{lemLongTrajectoriesSuffice}
If $\xi(x)\neq-\infty$ then for all $T>0$ and $r<\xi(x)$ there is $s>T$ and $u\in\mathcal{U}$ such that
$$\varphi(s,x,u)=e^{s\alpha} x$$
with $\alpha> r$.
\end{lemma}
\begin{proof}
By the definition of $\xi(x)$, there is $t>0$ and $u\in\mathcal{U}$ such that
$$\varphi(t,x,u)=e^{t\alpha}x$$
with $\alpha>r$. We can assume, without loss in generality, that $u$ as a cyclic function of period $t$. Then, by the linearity of the system and the co-cycle property, for all $k\in\mathbb{N}^*$,
$$\varphi(kt,x,u)=e^{kt\alpha}x.$$
Thus, it suffices to choose $s=kt$ for some sufficiently large $k$.
\end{proof}
As a consequence of this Lemma, if $\xi(x)\neq-\infty$ then there is a sequence of times $t_i>0$ and a sequence of controls $u_i\in\mathcal{U}$ such that $t_i\rightarrow+\infty$ and
$$\varphi(t_i,x,u_i)=e^{t_i\alpha_i}x$$
with $\alpha_i\rightarrow \xi(x)$.

Now we turn our attention to the projected system $\ww{\Sigma}$. Since this system is assumed to be accessible, from Proposition \ref{propExistenceOfInvControlSet} there is some (not necessarily unique) invariant control $C\subset S^{d-1}$, satisfying $\intr(C)\neq\emptyset$ and $\intr(C)\subset\oo^+(x)$ for all $\ww{x}\in C$. Furthermore, $C$ is a compact set. We fix one such invariant control set for the rest of this paper. Note that if $\ww{x}\in \intr(C)$ then $\xi(x)\neq-\infty$, since $\ww{x}\in\ww{\oo}^+(\ww{x})$.

\subsection{A particular case}
\label{sectionParticularCase}

Before going to the general case, we will assume that there is $\ww{x}\in \intr(C)$ such that $\xi(x)=0$ and $\Vert x\Vert=1$. This is assumed for all of the results in this subsection, and we later show how the general case can be reduced to this. This hypothesis implies that if $\alpha x\in\oo^+(x)$ with $\alpha>0$ then $\alpha\le 1$.

Using the inclusion $S^{d-1}\subset\mathbb{R}^d$, one can also consider $C$ as a subset of $\mathbb{R}^d$. The next results deal with this set, but under the dynamics of $\Sigma$ rather than $\ww{\Sigma}$. To avoid confusion, we will use the notation $S_C=C$ to reinforce that it is being considered as a subset of $\mathbb{R}^d$. Equivalently, $S_C$ is the set of all $y\in\pi^{-1}(C)$ with $\Vert y\Vert =1$. Since $C$ is compact, then so is $S_C$.

\begin{lemma}
\label{contractionAtBoundedCost}
There is $\epsilon>0$ such that if $y\in S_C$ then there is $\alpha>\epsilon$ such that $\alpha x\in\oo^+(y)$.
\end{lemma}
\begin{proof}
We first claim that there is $T>0$ such that $\widehat x\in\widehat\oo_{[0,T]}(\widehat y)$ for all $\widehat y\in C$.

In fact, let $r<0$ sufficiently small so that $\widehat\oo_{[r,0]}(\widehat x)\subset \intr(C)$. Since the system is locally accessible, there is $\widehat z\in \intr(\widehat\oo_{[r,0]}(\widehat x))$. In particular, $\widehat z\in \intr(C)$. For each $\widehat y\in C$, there is $v\in\mathcal{U}$ and $s>0$ such that
$$\widehat\varphi(s,\widehat y,v)=\widehat z.$$
Then, there is a neighborhood $V$ of $\widehat y$ such that
$$\widehat\varphi(s,V,v)\subset\widehat\oo_{t\in[r,0]}(\widehat x),$$
which implies that every element of $V$ can reach $\widehat x$ in a time no greater than $s_V=r+s$. We can construct one such neighborhood for each $\widehat y\in C$, and, since $C$ is compact, it can be covered by a finite amount of them, say, $V_1,V_2,...,V_k$. Then, we can take $T=\max(s_{V_1},s_{V_2},...,s_{V_k})$.

Thus, for all $y\in S_C$ there is $u\in\mathcal{U}$, $t\in[0,T]$ and $\alpha>0$ such that $\varphi_u^t(y)=\alpha x$. The compactness of $S_C,\mathcal{U}$ and $[0,T]$ ensures the existence of $\epsilon>0$ satisfying the Lemma.
\end{proof}

\begin{lemma}
\label{neighborhoodOfX}
There is a neighborhood $V\subset \intr(C)$ of $\widehat x$ and $L,\epsilon>0$ such that for every $z\in\pi^{-1}(V)\cap S_C$ there is $\alpha z\in\oo^+(x)$ with $L>\alpha>\epsilon$.
\end{lemma}
\begin{proof}
Let $s>0$ sufficiently small so that $\ww{\oo}_{[0,s]}(\ww{x})\subset \intr(C)$ and let $\ww{z}\in \intr(\ww{\oo}_{[0,s]}(\ww{x}))$. Since $\ww{z}\in \intr(C)$, there is $t>0$ and $u\in\mathcal{U}$ such that $\ww{\varphi}(t,\ww{z},u)=\ww{x}$. Then, $\widehat{x}\in\intr(\ww{\varphi}_u^t(\ww{\oo}_{[0,s]}(\ww{x})))\subset\intr(\ww{\oo}_{[0,t+s]}(\ww{x}))$. Thus, there is a neighborhood $V$ of $\ww{x}$ and $T=t+s>0$ such that $V\subset\ww{\oo}_{[0,T]}(\ww{x})$. This combined with compactness of $\mathcal{U}$ implies the Lemma.
\end{proof}

Let
$$S=\{x\in\mathbb{R}^d;\Vert x\Vert=1\}$$
the unitary sphere seen as a subset of $\mathbb{R}^d$.

\begin{proposition}
\label{boundedOrbit}
The orbit
$$\oo^+(S)=\bigcup_{y\in S}\oo^+(y)\subset\mathbb{R}^d$$
is bounded.
\end{proposition}
\begin{proof}
We first show that $\oo^+(x)$ is bounded. Assume, by contradiction, otherwise. Then there is a sequence $(y_i)_{i\in\mathbb{N}}\subset\oo^+(x)$ such that $\Vert y_i\Vert\rightarrow +\infty$. This sequence satisfies $y_i\in \pi^{-1}(C)$ for all $i\in\mathbb{N}$, since $C$ is invariant and $\widehat x\in C$. Thus, $\frac{y_i}{\Vert y_i\Vert}\subset S_C$ for all $i$. From Lemma \ref{contractionAtBoundedCost}, there is $\epsilon>0$ and a sequence $(a_i)_{i\in\mathbb{N}}$ such that $a_i>\epsilon$ and
$$a_ix\in\oo^+\left(\frac{y_i}{\Vert y_i\Vert}\right),$$
for every $i\in\mathbb{N}$. Then,
$$a_i\Vert y_i\Vert x\in\oo^+(y_i),$$
and, since $y_i\in\oo^+(x)$,
$$a_i\Vert y_i\Vert x\in\oo^+(x).$$
Then,
$$a_i\Vert y_i\Vert\le 1$$
for all $i\in\mathbb{N}$, which is a contradiction since $a_i>\epsilon$ and $\Vert y_i\Vert\rightarrow+\infty$. Thus proving the claim.

Now, from Lemma \ref{neighborhoodOfX}, there is a neighborhood $V\subset \intr(C)$ of $\ww{x}$ and $L_1,\epsilon_1>0$ such that for all $z\in S\cap\pi^{-1}(V)$ there is $s>0$ such that $L_1>s>\epsilon_1$ and $sz\in\oo^+(x)$.

Since $S$ is compact, there are $\delta,L_2,\epsilon_2>0$ such that for all $y\in S$, there are $r>0$ and $z\in S\cap\pi^{-1}(V)$ satisfying $L>r^{-1}>\epsilon_2$ and $x+\delta y=rz$, or, equivalently, $r^{-1}(x+\delta y)=z$. Furthermore, since $z\in S\cap \pi^{-1}(V)$, there is $s>0$ such that $L_1>s>\epsilon_1$ and $sz=sr^{-1}(x+\delta y)\in\oo^+(x)$. Thus, taking $L=L_1L_2>0$ and $\epsilon=\epsilon_1\epsilon_2$, we have that for all $y\in S$ there is $\alpha>0$ such that $L>\alpha>\epsilon$ and $\alpha(x+\delta y)\in\oo^+(x)$.

Assume by, contradiction, that $\oo^+(S)$ is not bounded. Then there is a sequence $(z_i)_{i\in\mathbb{N}}$  such that $\Vert z_i\Vert\rightarrow+\infty$ and each term of the sequence can be written as $z_i=\varphi_{u_i}^{t_i}y_i$ with $y_i\in S$, $u_i\in\mathcal{U}$ and $t_i\ge 0$. From the previous argument, there is a sequence $(\alpha_i)_{i\in\mathbb{N}}$ with $L>\alpha_i>\epsilon$ such that
$$\alpha_i(x+\delta y_i)\in\oo^+(x).$$
Then,
$$\varphi_{u_i}^{t_i}\alpha_i(x+\delta y_i)\in\oo^+(x).$$
But
$$\Vert \varphi_{u_i}^{t_i}\alpha_i(x+\delta y_i)\Vert=\Vert\alpha_i\varphi_{u_i}^{t_i}x+\alpha_i\delta z_i\Vert\ge \alpha_i\delta \Vert z_i\Vert-\alpha_i\Vert\varphi_{u_i}^{t_i}x\Vert.$$
The sequence $\alpha_i$ is bounded from bellow by $\epsilon$ and $\Vert z_i\Vert\rightarrow+\infty$, thus $\alpha_i\delta\Vert z_i\Vert\rightarrow+\infty$. However, $\alpha_i$ is also bounded from above by $L$ and $\Vert\varphi_{u_i}^{t_i}x\Vert$ is bounded since $\oo^+(x)$ is bounded. Thus $\alpha_i\Vert \varphi_{u_i}^{t_i}x\Vert$ is bounded, and, then
$$\lim_{t\rightarrow+\infty}\alpha_i\delta \Vert z_i\Vert-\alpha_i\Vert\varphi_{u_i}^{t_i}x\Vert=+\infty.$$
Consequently,
$$\Rightarrow\lim_{t\rightarrow+\infty}\Vert\varphi_{u_i}^{t_i}\alpha_i(x+\delta y_i)\Vert=\lim_{t\rightarrow+\infty}\Vert\alpha_i\varphi_{u_i}^{t_i}x+\alpha_i\delta z_i\Vert=+\infty,$$
which is a contradiction since $\oo^+(x)$ is bounded.
\end{proof}

We now start working with star sets (Definition \ref{definitionStarSet}). Due to the linearity of the system, if $W$ is a star set then $\oo_t(W)$ and $\oo^+(W)$ are star sets for all $t\in\mathbb{R}$.

Consider the set
$$B=[0,1]S=\{\alpha x;\alpha\in[0,1],x\in S\}=\{x\in\mathbb{R}^d;\Vert x\Vert\le 1\}.$$
Note that $B$ is a star set, and thus, so is $\oo^+(B)$. Furthermore,
$$\oo^+(B)=[0,1]\oo^+(S).$$
Thus, from Proposition \ref{boundedOrbit}, $\oo^+(B)$ is bounded. It is also an invariant set, since $\oo^+(\oo^+(B))=\oo^+(B)$. Let
$$D_0=\overline{\oo^+(B)}.$$
Then $D_0$ is also an invariant star set, and, since $\oo^+(B)$ is bounded, $D_0$ is compact.

Note that $0\in \intr(B)\subset \intr(\oo^+(B))$, thus $0\in \intr(D_0)$. Therefore, if $\partial D_0$ denotes the topological boundary of this set, then $0\not\in\partial D_0$.

Define the set
$$D=\bigcap_{t>0}\oo_tD_0.$$
Since $\mathcal{U}$ is compact, then $\oo_tD_0$ is compact for each $t>0$. Furthermore, since $D_0$ is invariant, then $\oo_tD_0\subset\oo_sD_0$ whenever $t>s>0$, and, thus, the intersection above is descending. Note that for all $t>0$, $\oo_tD_0\neq\emptyset$, thus $D$ is a nonempty compact set. We will prove that $D$ is nontrivial, that is, $D\neq \{0\}$. For this, it suffices to prove that $D\cap\partial D_0\neq\emptyset$, since $0\not\in\partial D_0$. We start with the following lemmas.

\begin{lemma}
Let $V\subset\mathbb{R}^d$ an open star set, and $K$ a compact set. If $K\subset V$ then there is $\alpha\in(0,1)$ such that $K\subset \alpha V$.
\end{lemma}
\begin{proof}
Since $V$ is an open star set, we can write
$$V=\bigcup_{t\in(0,1)}tV,$$
thus $\{tV;t\in(0,1)\}$ is an open covering of $K$. Therefore, there are $t_1,t_2,...,t_n\in(0,1)$ such that
$$K\subset\bigcup_{i=1}^nt_nV,$$
and, then, taking $\alpha=\max(t_1,t_2,...,t_n)\in(0,1)$, we have
$$K\subset\alpha V$$
\end{proof}

\begin{lemma}
For all $t>0$, $\oo_t(D_0)\cap\partial D_0\neq\emptyset$.
\end{lemma}
\begin{proof}
$D_0$ is invariant, and, thus $\oo_t(D_0)\subset D_0$ for all $t>0$. Assume, by contradiction, that $\oo_t(D_0)\cap\partial D_0=\emptyset$ for some $t>0$. Then, $\oo_t(D_0)\subset \intr(D_0)$. The set $\oo_t(D_0)$ is compact since $D_0$ and $\mathcal{U}$ are compact, and $\intr(D_0)$ is a star set since $D_0$ is a star set and $0\in \intr(D_0)$. Then, from the previous Lemma there is $\alpha\in(0,1)$ such that
$$\oo_t(D_0)\subset \alpha D_0.$$
Using the property of the cocycle, we get
$$\oo_{kt}(D_0)\subset \alpha^kD_0$$
for all $k\in\mathbb{N}$.

Since $D_0$ is compact, there is $R>0$ such that $R\ge \Vert y\Vert$ for all $y\in D_0$. Furthermore, $x\in B\subset D_0$. The previous equality then implies that, for any $u\in\mathcal{U}$ and $k\in\mathbb{N}$,
$$\Vert\varphi_u^{kt}(x)\Vert\le\alpha^kR.$$
We can write $\alpha=e^{-\epsilon}$ with $\epsilon>0$. Then:
$$\Vert\varphi_u^{kt}(x)\Vert\le e^{-\epsilon k}R.$$

Since $\mathcal{U}$ is compact, there is $L>0$ such that for all $u\in\mathcal{U}$ and $s\in[0,t]$,
$$\Vert \varphi_u^s\Vert<L.$$

Now, since $\xi(x)=0$, from the Lemma \ref{lemLongTrajectoriesSuffice}, there are sequences $u_i\in\mathcal{U}$, $s_i>0$ and $\alpha_i>0$ such that $s_i\rightarrow +\infty$,
$$\varphi_{u_i}^{s_i}x=e^{s_i\alpha_i}x$$
and $\alpha_i\rightarrow 0$. This implies that
$$\lim_{i\rightarrow +\infty}\frac{1}{s_i}\ln\Vert \varphi_{u_i}^{s_i}x\Vert=\lim_{i\rightarrow +\infty}\alpha_i=0,$$
since $\Vert x\Vert=1$.

Each $s_i$ can be written as $s_i=k_it+r_i$ with $r_i\in[0,t)$ and $k\in\mathbb{N}$. Furthermore, since $s_i\rightarrow+\infty$ then $k_i\rightarrow+\infty$ and $\frac{s_i}{k_i}\rightarrow t$. Then, for each $i\in\mathbb{N}$ we can write
$$\varphi_{u_i}^{s_i}=\varphi_{v_i}^{r_i}\circ\varphi_{u_i}^{k_it},$$
where $v_i\in\mathcal{U}$ is defined by $v_i(s)=u_i(s+k_it)$. Then,
$$\frac{1}{s_i}\ln\Vert\varphi_{u_i}^{s_i}x\Vert\le\frac{1}{s_i}\ln(\Vert\varphi_{v_i}^{r_i}\Vert\Vert\varphi_{u_i}^{k_it}x\Vert)<\frac{1}{s_i}\ln(Le^{-\epsilon k_i}R)$$
$$=\frac{1}{s_i}(\ln(L)-\epsilon k_i+\ln(R))=\frac{ln(L)+ln(R)}{s_i}-\frac{\epsilon k_i}{s_i}.$$
However,
$$\lim_{i\rightarrow +\infty}\frac{ln(L)+ln(R)}{s_i}-\frac{\epsilon k_i}{s_i}=-\frac{\epsilon}{t}<0,$$
which contradicts the previous limit.
\end{proof}

\begin{corollary}
\label{corollaryNontrivial}
$$D\cap\partial D_0\neq\emptyset$$
\end{corollary}
\begin{proof}
Note that
$$D\cap\partial D_0=\bigcap_{t>0}(\oo_tD_0\cap\partial D_0).$$
The intersection on the right is a descending intersection of compact sets, and, by the previous Lemma, they are all nonempty. Therefore $D\cap \partial D_0\neq\emptyset$.
\end{proof}

Therefore $D$ is not a trivial set. We now show that it satisfies $\oo_t(D)=D=e^{0t}D$ for all $t>0$:

\begin{proposition}
\label{caseXiZero}
For all $t>0$, $\oo_t(D)=D$.
\end{proposition}
\begin{proof}
Note that
$$\oo_t(D)=\oo_t\left(\bigcap_{s>0}\oo_s(D_0)\right)=\bigcap_{s>0}\oo_t\left(\oo_s(D_0)\right)=\bigcap_{s>t}\oo_s(D_0).$$
However, the intersection above is descending. In fact, if $s_1,s_2\in\mathbb{R}$ are such that $s_1>s_2$ then
$$\oo_{s_1}(D_0)=\oo_{s_2}\left(\oo_{s_1-s_2}(D_0)\right)\subset\oo_{s_2}(D_0),$$
since $D_0$ is invariant. Therefore,
$$\bigcap_{s>t}\oo_s(D_0)=\bigcap_{s>0}\oo_s (D_0)=D,$$
and, thus,
$$\oo_t(D)=D.$$
\end{proof}

Taking $R=0$, we have, from Proposition \ref{caseXiZero}, that
$$\oo_t(D)=e^{tR}D$$
for all $t>0$. From Corollary \ref{corollaryNontrivial}, $D$ is also not trivial, and, thus, $D$ is eigenset.

\subsection{The general case}
\label{generalCase}
On Section \ref{sectionParticularCase} we proved that $\Sigma$ admits an eigenset if there is $x\in\pi^{-1}(\intr(C))$ with $\xi(x)=0$. We will now generalize these results for the general case (still assuming accessibility of $\ww{\Sigma}$).

Denoting
$$F:U\rightarrow\mathfrak{gl}(d,\mathbb{R})$$
$$u\rightarrow A+\sum_{i=1}^mu_iB_i,$$
we have that the equation for $\Sigma$ can be rewritten as
$$\dot x(t)=F(u(t))x(t).$$
For each $r\in\mathbb{R}$ we define a new bilinear system $\Sigma_r$ by
$$\dot x(t)=(F(u(t))-rId)x(t)=(A-rId)x+\sum_{i=1} ^mu_iB_i;~u\in\mathcal{U},$$
where $Id$ denotes the identity matrix in $\ggl$.

If $\varphi_r$ denotes the solution of this new system, then for all $t\in\mathbb{R},x\in\mathbb{R}^d$ and $u\in\mathcal{U}$,
\begin{equation}
\label{equalityr}
\varphi_r(t,x,u)=e^{-tr}\varphi(t,x,u).
\end{equation}
In fact, differentiating the curve on the right we have:
$$\frac{d}{dt}e^{-tr}\varphi(t,x,u)=-re^{-tr}\varphi(t,x,u)+e^{-tr}F(u(t))\varphi(t,x,u)$$
$$=(F(u(t))-rId)e^{-tr}\varphi(t,x,u),$$
and, furthermore, $e^{-0r}\varphi(t,x,u)=x$, thus $e^{-tr}\varphi(t,x,u)$ is solution for $\Sigma_r$.

From this equality one sees that $\Sigma$ and $\Sigma_r$ project into the same system $\ww{\Sigma}=\ww{\Sigma_r}$ in $S^{d-1}$. In particular, $\ww{\Sigma_r}$ is also accessible in the sphere, and, if $C\subset S^{d-1}$ is an invariant control set for $\ww{\Sigma}$ then it is also an invariant control set for $\ww{\Sigma_r}$.

Since $\Sigma_r$ is also a bilinear control system, we can define $\xi_r:\mathbb{R}^d\setminus\{0\}\rightarrow \mathbb{R}\cup\{+\infty,-\infty\}$ for $\Sigma_r$ similarly to how $\xi$ was defined for $\Sigma$. From the equality \ref{equalityr} we can also see that
$$\xi_r(x)=\xi(x)-r$$
for all $x\in\mathbb{R}^d\setminus\{0\}$.

Now fix $x\in\pi^{-1}(\intr(C))$ with $\Vert x\Vert=1$. We have $\xi(x)\neq+\infty$ due to Proposition \ref{propBoundedSystem}, and $\xi(x)\neq-\infty$ since $\ww{x}\in\ww{\oo}^+(\ww{x})$, thus $\xi(x)\in\mathbb{R}$. Denote $R=\xi(x)$. Then $\xi_R(x)=0$, and the system $\Sigma_R$ satisfies the hypothesis of the previous section.

We will use $\oo^R$ to denote the orbits in the system $\Sigma_R$.

\begin{proposition}
$\xi$ is constant in $\pi^{-1}(\intr(C))$, furthermore, for all $y\in\mathbb{R}^d\setminus\{0\}$, $\xi(y)\le R$.
\end{proposition}
\begin{proof}
Let $y\in\mathbb{R}^d\setminus\{0\}$ be arbitrary. From Lemma \ref{boundedOrbit}, the orbit $(\oo^R)^+(\frac{y}{\Vert y\Vert})$ is bounded, which, by linearity of $\Sigma_R$, implies that $(\oo^R)^+(y)$ is bounded. This, combined with Lemma \ref{lemLongTrajectoriesSuffice}, implies that $\xi_R(y)\le 0$, thus $\xi(y)\le R=\xi(x)$. If $y\in \pi^{-1}(\intr(C))$ then a similar argument shows that $\xi(x)\le\xi(y)$, and, thus, $\xi(x)=\xi(y)$.
\end{proof}

Therefore, the constant $R$ does not depend on the point $x\in \pi^{-1}(\intr(C))$ chosen. In fact, from the previous proposition it can also be seen that if $C_2$ is another invariant control set and $y\in\pi^{-1}(\intr(C_2))$ then $\xi(x)=\xi(y)$, so that $R$ depends only on the control system $\Sigma$.

\begin{theorem}
\label{mainTheorem}
If $\Sigma$ is accessible in $S^{d-1}$ then there is a compact subset $D\subset\mathbb{R}^d$ that is nontrivial in the sense that $D\not\subset\{0\}$, such that for all $t>0$
$$\oo^+(D)=e^{tR}D.$$
\end{theorem}
\begin{proof}
From proposition \ref{caseXiZero}, there is a compact, nontrivial set $D\subset\mathbb{R}^d$ such that for all $t>0$
$$(\oo^R)_t(D)=D.$$
Using equality \ref{equalityr}, we have that
$$\oo_t(D)=e^{tR}D$$
for all $t>0$.
\end{proof}

The previous Theorem proves the existence of an eigenset $D$. In general, $D$ needs not to be the only set with this property (this is shown later in Example \ref{example1}).

In the next proposition we see how a point $x\in\mathbb{R}^n$ satisfying $R\in S_x$ can be used to construct an eigenset.

\begin{proposition}
\label{propConstruct}
    Assume that $\ww{\Sigma}$ is a accessible, let $R$ as previously and assume that there is $x\in\mathbb{R}^n\setminus\{0\}$ such that $R\in S_x$. Then $\overline{(\oo^R)^+(x)}$ is an eigenset.
\end{proposition}
\begin{proof}
    Since $R\in S_x$, there are $u_0\in\mathcal{U}$ and $t_0>0$ such that
    $$\varphi(t_0,x,u_0)=e^{t_0R}x.$$
    We can assume, without loss in generality, that $u_0$ is cyclic with period $t_0$, and, then,
    $$\varphi(kt_0,x,u_0)=e^{kt_0R}x,$$
    or, equivalently,
    $$\varphi_R(kt_0,x,u_0)=x$$
    for all $k\in\mathbb{N}$.

    Let $\alpha=\Vert x\Vert$ and $S=\{x\in\mathbb{R}^n;\Vert x\Vert=1\}$ the unity sphere seen as a subset of $\mathbb{R}^n$. Note that
    $$(\oo^R)^+(x)\subset(\oo^R)^+(\alpha S)=\alpha(\oo^R)^+(S).$$
    From proposition \ref{boundedOrbit}, $(\oo^R)^+(S)$ is bounded, therefore $(\oo^R)^+(x)$ is bounded, and, consequently, $\overline{(\oo^R)^+(x)}$ is compact.

    Denote $D=\overline{(\oo^R)^+(x)}$, and let $t>0$ be arbitrary. $D$ is invariant by $\Sigma_R$, therefore,
    $$\oo_t(D)=e^{tR}((\oo^R)_t(D))\subset e^{tR}D.$$
    For the other inclusion, since $\oo_t(D)$ is closed it suffices to prove that $e^{tR}y\in \oo_t(D)$ for all $y\in(\oo^R)^+(x)$. Taking an arbitrary $y\in(\oo^R)^+(x)$, there are $u\in\mathcal{U}$ and $s>0$ such that
    $$y=\varphi_R(s,x,u).$$

    Note that, for all $k\in\mathbb{N}$ there is $v_k\in\mathcal{U}$ such that
    $$y=\varphi_R(s,x,u)=\varphi_R(s,\varphi_R(kt_0,x,u_0),u)=\varphi_R(s+kt_0,x,v_k),$$
    where $t_0$ and $u_0$ are as previously, and $v_k$ is defined by
    $$v_k(r)=\left\{\begin{array}{ll}
        u_0(r) & \text{ if } r\le kt_0\\
        u(r) & \text{ if } r>kt_0
    \end{array}\right.$$
    Then we can assume, without loss in generality, that $s>t$, since, otherwise, we could instead take $s+kt_0$ for some sufficiently large $k\in\mathbb{N}$. Therefore,
    $$y=\varphi_R(t,z,v),$$
    where $z=\varphi_R(s-t,x,u)\in D$ and $v$ is defined by $v(r)=u(r+s-t)$ for all $r\in\mathbb{R}$. Thus,
    $$y=e^{-tR}\varphi(t,z,v)$$
    $$\Rightarrow e^{tR}y=\varphi(t,z,v)\in\oo_t(D).$$
    Since $y\in(\oo^R)^+$ is arbitrary and $\oo_t(D)$ is closed, this shows that $e^{tR}D\subset\oo_t(D)$ and, therefore, $e^{tR}D=\oo_t(D)$.
\end{proof}
In general the set $D$ obtained in the previous proposition needs not to be a star set. The star shape can be retrieved by instead considering the set
$$[0,1]\times D=[0,1](\overline{(\oo^R)^+(x)})=\overline{(\oo^R)^+([0,1]x)}.$$

In the next example we see that the eigenset needs not to be unique.

\begin{example}
\label{example1}
    Let $\Sigma$ be the bilinear control system defined by
    $$\dot x(t)=Ax(t)+u(t)Bx(t)$$
    with $U=[-1,1]\subset\mathbb{R}$ and
    $$A=\begin{pmatrix}
        -1 & 1\\
        1 & 1
    \end{pmatrix},$$
    $$B=\begin{pmatrix}
        1 & 1\\
        1 & -1
    \end{pmatrix}.$$
    We can rewrite the system as
    $$\dot x(t)=M(u(t))x(t)$$
    where, for each $u_0\in[-1,1]$,
    $$M(u_0)=A+u_0B=\begin{pmatrix}
        -1+u_0 & 1+u_0\\
        1+u_0 & 1-u_0
    \end{pmatrix}.$$
    Calculating $M$ in the extreme points of the interval we have
    $$M(-1)=\begin{pmatrix}
        -2 & 0\\
        0 & 2
    \end{pmatrix},$$
    $$M(1)=\begin{pmatrix}
        0 & 2\\
        2 & 0
    \end{pmatrix},$$
    and all other values of $M$ are convex combinations of these two values. Any control $u\in\mathcal{U}$ can be approximated by a sequence of controls $u_n$ with images in $\{-1,1\}$ such that $\varphi_{u_n}^t\rightarrow\varphi_{u}^t$ for all $t\in\mathbb{R}$.

    The system above is accessible in $S^1$, thus it is in the hypothesis of Theorem \ref{mainTheorem}. We will construct a set with property \ref{eigenvectorLikeProperty}.

    We start by calculating $R$. In this example we are able to bound $\xi$ from above using the norm of the operators $M(t)$ for $t\in[-1,1]$. Consider the euclidean norm in $\mathbb{R}^2$:
    $$\Vert(x,y)\Vert=x^2+y^2$$
    and the operator norm in ${\rm gl}(2,\mathbb{R})$:
    $$\Vert T\Vert=\sup\{\Vert Tx\Vert;x\in\mathbb{R}^2\text{ and }\Vert x\Vert=1\}.$$
    Note that
    $$\Vert M(-1)\Vert=\Vert M(1)\Vert =2.$$
    Furthermore, for all $t\in(-1,1)$, $M(t)$ is a convex combination of $M(-1),M(1)$, thus $\Vert M(t)\Vert\le 2$. Now let $x\in\mathbb{R}^2\setminus\{0\}$ and $u\in\mathcal{U}$ be arbitrary. Consider the curve
    $$c:\mathbb{R}\rightarrow\mathbb{R}$$
    $$t\rightarrow \ln\Vert\varphi(t,x,u)\Vert.$$
    Note that
    $$\frac{d}{dt}c(t)=\frac{1}{\Vert\varphi(t,x,u)\Vert}\frac{\langle \varphi(t,x,u),M(u(t))\varphi(t,x,u)\rangle}{\Vert\varphi(t,x,u)\Vert}$$
$$\le \frac{\Vert\varphi(t,x,u)\Vert\Vert M(u(t))\varphi(t,x,u)\Vert}{\Vert\varphi(t,x,u)\Vert^2}\le\Vert M(u(t))\Vert\le 2.$$
    Thus, by the fundamental theorem of calculus,
    $$c(t)-c(0)\le 2t$$
    for all $t>0$. Therefore, for all $x\in\mathbb{R}^2-\{0\},u\in\mathcal{U}$ and $t>0$,
    $$\ln\left(\frac{\Vert\varphi(t,x,u)\Vert}{\Vert x\Vert}\right)\le 2t.$$
    Consequently, $\xi(x)\le 2$ for all $x\in\mathbb{R}^2-\{0\}$.
    
    On the other hand, considering the point $(1,1)$ and the control defined by $u(t)=1$ for all $t\in\mathbb{R}$, we have
    $$\varphi(t,(1,1),u)=(e^{2t},e^{2t})$$
    for all $t\in\mathbb{R}$. Thus, $\xi(1,1)=2=R$.

    From proposition \ref{propConstruct}, taking $C=\{(\alpha,\alpha);0\le\alpha\le 1\}$, the set $D=\overline{(\oo^2)^+(C)}$ is an eigenset. To calculate this set, note that $(\oo^2)^+(C)$ is limited by the semi-lines $\{(0,t);t\ge 0\}$ and $\{(t,t);t\ge 0\}$, as no vector field in $\Sigma_2$ points to the left, right of these lines, respectively. It is also bounded by the segment $\{(t,1);t\in[0,1]\}$, as no vector field points up in this segment. Thus, $\overline{(\oo^2)^+(C)}$ is contained in the triangle with vertices in $(0,0),(0,1)$ and $(1,1)$.

    On the other hand, taking the constant control defined by $u(t)=-1$ for all $t$, we have
    $$\varphi_2(t,(1,1),u)=\exp(t(M(-1)-2Id))(1,1)=\begin{pmatrix}
        e^{-4t} & 0\\
        0 & 1
    \end{pmatrix}\begin{pmatrix}
        1\\1
    \end{pmatrix}=(e^{-t4},1),$$
    thus, the semi-open segment $\{(t,1);0<t\le 1\}$ is contained in $(\oo^2)^+(C)$, and, since this set is star shaped, it contains the interior of the triangle with points in $(0,0),(0,1)$ and $(1,1)$. Therefore, $D$ coincides with this triangle.

    Now consider the point $(0,1)$. Note that
    $$\varphi(t,(0,1),M(-1))=e^{tM(-1)}(0,1)=e^{2t}(0,1),$$
    therefore, $R\in S_{(0,1)}$ and taking $C_2=\{(0,t);0\le t\le 1\}$ we can apply Proposition \ref{propConstruct} to get another set $D_2=\overline{(\oo^2)^+(C)}$ with property \ref{eigenvectorLikeProperty}. Using similar arguments as before, one shows that $D_2$ coincides with the triangle with vertices in $(0,0),(0,1),\left(\frac{1}{2},\frac{1}{2}\right)$, and is thus distinct from $D$.

    Furthermore, if $\alpha,\beta$ are real numbers, the set $(\alpha D)\cup(\beta D_2)$ will also satisfy property \ref{eigenvectorLikeProperty}, since, for all $t>0$,
    $$\oo_t\left((\alpha D)\cup(\beta D_2)\right)=(\alpha \oo_t(D))\cup(\beta\oo_t(D_2))$$
    $$=(e^{tR}\alpha D)\cup(e^{tR}\beta D_2)=e^{tR}\left((\alpha D)\cup(\beta D_2)\right).$$
    This yields an infinite family of eigensets.
\begin{figure}[h]
\hfill
\subfigure{\includegraphics[width=5cm]{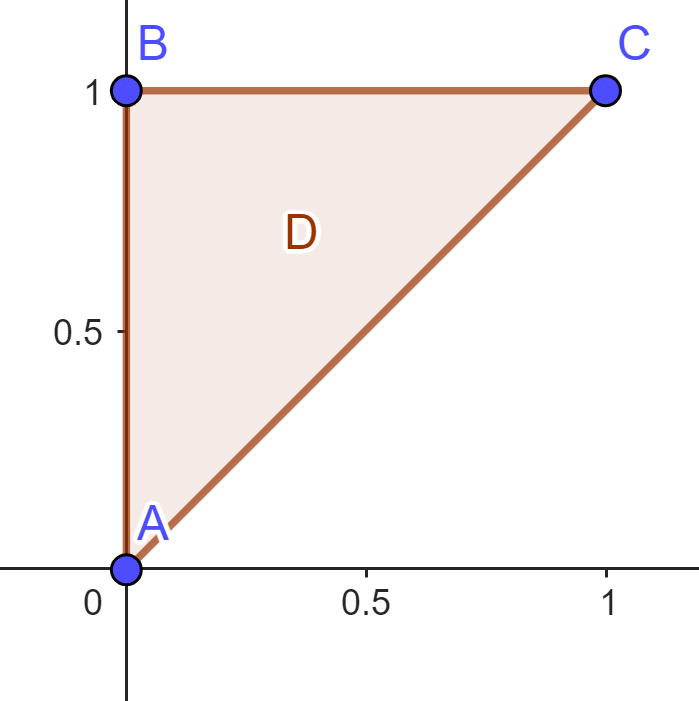}}
\hfill
\subfigure{\includegraphics[width=5cm]{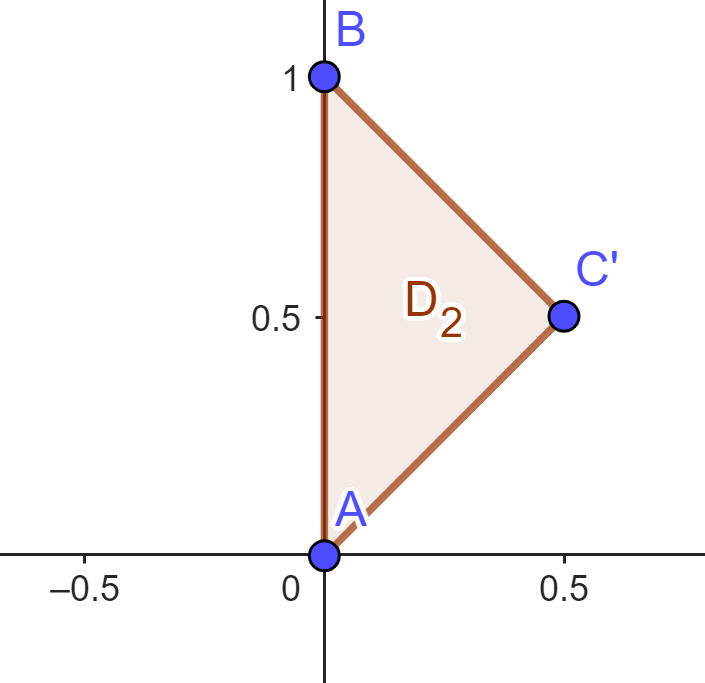}}
\hfill
\caption{The sets $D$ and $D_2$ from example \ref{example1}.}
\end{figure}
\end{example}

\pagebreak
\begin{example}
\label{example2}
    Let $\Sigma$ be the bilinear control system defined by
    $$\dot x(t)=Ax(t)+u(t)Bx(t)$$
    with $U=[-1,1]\subset\mathbb{R}$ and
    $$A=\begin{pmatrix}
        1 & -2\\
        2 & 1
    \end{pmatrix}=Id+\begin{pmatrix}
        0 & -2\\
        2 & 0
    \end{pmatrix},$$
    $$B=\begin{pmatrix}
        -1 & -2\\
        2 & -1
    \end{pmatrix}=-Id+\begin{pmatrix}
        0 & -2\\
        2 & 0
    \end{pmatrix}.$$
    Similar to the previous example, we define $M(u_0)=A+u_0B$ for all $u_0\in[-1,1]$. In this case we have
    $$M(u_0)=(1-u_0)Id+(1+u_0)\begin{pmatrix}
        0 & -2\\
        2 & 0
    \end{pmatrix},$$
    and one can prove that for all $u\in\mathcal{U}$ and $t\in\mathbb{R}$,
    \begin{equation}
        \label{solutionExample2}
        \varphi_u^t=e^{t-I_u(t)}R(2t+2I_u(t))
    \end{equation}
    where
    $$I_u(t)=\int_{0}^tu(s)ds$$
    and $R(r)$ denotes the counter clockwise rotation by $r$ radians. Thus, $R=2$ and every $x\in\mathbb{R}^2\setminus\{0\}$ satisfies $R\in S_x$. In particular, taking the point $(1,0)$ and $C=\{(t,0);0\le t\le 1\}$ we have by Proposition \ref{propConstruct} that $D=\overline{(\oo^2)^+(C)}$ satisfies property \ref{eigenvectorLikeProperty}. From equation \ref{solutionExample2} one sees that $(\oo^2)^+(C)$ is the set under the polar curve $\{e^{-t}(\cos(2t),\sin(2t)),0\le t\le \pi\}=\{e^{-\frac{1}{2}t}(\cos(t),\sin(t)),0\le t\le 2\pi\}$.
    
    \begin{figure}[h]
        \centering
        \includegraphics[width=5cm]{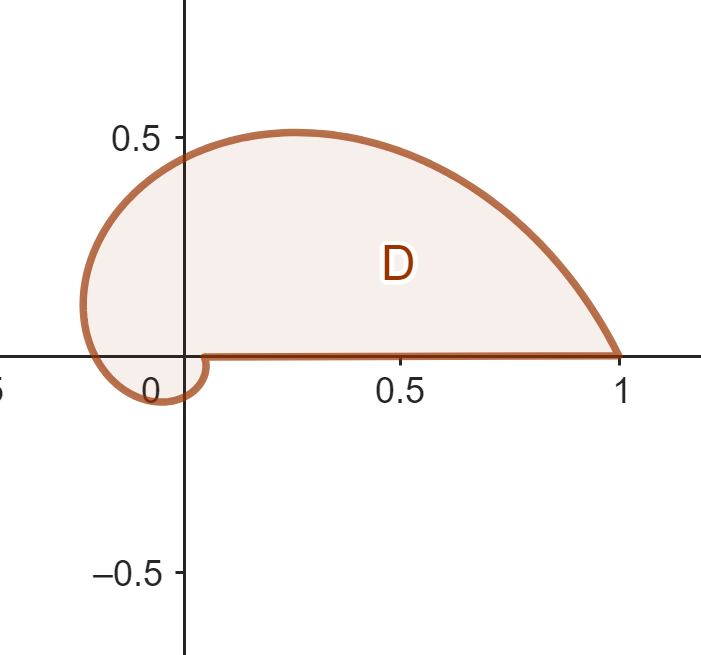}
        \caption{Set $D$ from example \ref{example2}}
    \end{figure}

\end{example}

\end{document}